\documentclass[12pt,oneside]{amsart}
\usepackage{amsfonts}
\usepackage{amsmath}
\usepackage{amssymb}
\usepackage{mathrsfs}
\usepackage{graphicx}
\usepackage{subfigure}
\usepackage{hyperref}
\usepackage{microtype}
\usepackage{enumerate}
\usepackage{epsfig}
\usepackage{cancel}
\usepackage{color}
\usepackage{cite}
\usepackage{makecell}
\usepackage{multirow}
\usepackage{footmisc}
\usepackage{wrapfig}
\usepackage{array}
\usepackage{titletoc}
\usepackage{bm}
\usepackage{times}
\hypersetup{
colorlinks=true,
linkcolor=blue,
filecolor=blue,
urlcolor=blue,
citecolor=blue,}

\usepackage{geometry}
\usepackage{multirow}
\newtheorem{theorem}{Theorem}[section]
\newtheorem{lemma}[theorem]{Lemma}

\newtheorem{proposition}[theorem]{Proposition}

\numberwithin{equation}{section}

\theoremstyle{remark}
\newtheorem{remark}{Remark}[section]

\allowdisplaybreaks[3]

\def\De{\Delta}

\def\na{\nabla}
\def\pat{\partial_t}

\def\lan{\langle}
\def\ran{\rangle}

\newcommand{\la}{\lambda}
\newcommand{\al}{\alpha}

\newcommand{\vth}{\vartheta}

\newcommand{\R}{\mathbb{R}}

\newcommand{\Z}{\mathbb{Z}}
\newcommand{\T}{\mathbb{T}}

\newcommand{\om}{\omega}

\newcommand{\n}[1]{\Vert #1\Vert }
\newcommand{\bn}[1]{\big \Vert #1 \big \Vert }
\newcommand{\bbn}[1]{\Big\Vert #1 \Big \Vert }
\newcommand{\lr}[1]{\left\{ #1 \right\} }

\newcommand{\lrs}[1]{\left( #1 \right) }

\newcommand{\pa}{\partial}

\begin{document}
\title[Quantitative Stability of the 2D Monotone Shear Flow]{Quantitative stability of the 2D Monotone shear flow for Boussinesq equation in a finite channel}

\author[]{Qionglei Chen}
\address[]{Institute of Applied Physics and Computational Mathematics, 100088 Beijing, China}
\email{chen\_qionglei@iapcm.ac.cn}

\author[]{Zhen Li}
\address[]{School of Mathematical Sciences, Key Laboratory of Mathematics and Complex Systems, Ministry of Education, Beijing Normal University, 100875 Beijing, China}
\email{lizhen@bnu.edu.cn}

\keywords{Monotone shear flow, Transition threshold, Finite channel, Non-slip boundary conditions}
\begin{abstract}
Neither natural nor laboratory laminar flows are perfectly steady. Instead, they are frequently highly unsteady, as illustrated by experimental studies on B\'{e}nard convection. In the paper, we investigate the transition threshold of the Boussinesq equations around a time-dependent monotone shear flow $(U(t,y),0)$ with a constant background temperature $a\in\R$. The analysis is performed in the finite channel $\T\times[0,1]$ with non-slip boundary condition. By means of the sharp resolvent estimates and space-time estimates, we establish that the Boussinesq system admits a globally stable solution around  the monotone shear flow, provided that the initial perturbation satisfies $\|u^{\mathrm{in}}\|_{H^2}\leq c\nu^{\frac12}, \|\lan D_x\ran \theta^{\mathrm{in}}\|_{L^2} \leq c\nu^{\frac56}$. Moreover, we derive the enhanced dissipation estimate of the vorticity and inviscid damping estimate of the velocity.
\end{abstract}

\maketitle
\section{Introduction}\label{sec:Introduction}
We study the 2D incompressible Boussinesq equations in a finite channel $\T\times [0,1]$:
\begin{equation}\label{Bou}
	\left\{
    \begin{aligned}
    &\pat v-\nu\De v + (v\cdot\na)v + \na P =\vartheta e_{2}, \\
    &\pat \vartheta-\mu \De \vth+(v\cdot\na)\vth=0,\\
    &\na\cdot v = 0,\\
    & v(0,x,y) = v^{\mathrm{in}}(x,y),
    \end{aligned}
	\right.
	\end{equation}
where $v(t,x,y)$ is the velocity field, $\vth(t,x,y)$ is the temperature, $P$ is the pressure and $e_{2}=(0,1)$ is the unit vector in the vertical direction. Moreover, $\nu$ and $\mu$ are the viscosity coefficient and the thermal diffusivity, respectively. We focus on the case $\nu=\mu$ for simplicity throughout this paper.

The Boussinesq system serves as a simplified model for the fluid dynamics of the atmospheric and oceanographic flows and plays an important role in studying Rayleigh-B\'{e}nard convection for laminar and turbulent flows \cite{Gill,Pedlosky,Drazin-Reid}. From a mathematical view, the 2D Boussinesq equations preserves some key features of the 3D Euer and Navier-Stokes equations. Thanks to the notable significance in mathematics and physics of this model, numerous scholars have directed their focus toward its theoretical exploration.

The transition threshold problem has been the focus of extensive research efforts in fluid dynamics. The Couette flow, among the simplest laminar flows, has attracted the attention of many mathematicians. On the domain $\T\times \R$, Deng-Wu-Zhang \cite{Deng-Wu-Zhang} established the asymptotic stability, with the initial data $(\om^{\mathrm{in}}, \theta^{\mathrm{in}})$ satisfying
$$\|\om^{\mathrm{in}}\|_{H^s}\leq c\nu^{\frac12}, \quad \|\theta^{\mathrm{in}}\|_{H^s}\leq c\nu, \quad \||D_x|^{\frac13} \theta^{\mathrm{in}}\|_{H^s}\leq c\nu^{\frac56},\quad s>1.$$
Zhang-Zi \cite{Zhang-Zi} improved the nonlinear stability results under the following initial data constraints:
$$\|\om^{\mathrm{in}}\|_{H^s}\leq c\nu^{\frac13}, \quad \|\theta^{\mathrm{in}}\|_{H^s}\leq c\nu^{\frac56}, \quad \||D_x|^{\frac13} \theta^{\mathrm{in}}\|_{H^s}\leq c\nu^{\frac23}, \quad s>7.$$
Subsequently, Niu-Zhao \cite{Niu-Zhao} refined the bound for the perturbed temperature from $\nu^{\frac56}$ to $\nu^{\frac23}$ and reduced the required regularity index from $s>7$ to $s>5$. For the finite channel $\T\times [-1,1]$ with the non-slip boundary condition, Masmoudi-Zhai-Zhao \cite{Masmoudi-Zhai-Zhao} showed the transition threshold is bounded above by $(\nu^{\frac12}, \nu^{\frac{11}{12}})$. For the infinite channel $\R\times [-1,1]$ under the same boundary condition, Liang-Wu-Zhai \cite{Liang-Wu-Zhai} determined the upper bound of the threshold as $(\nu^{\frac12}, \nu^{\frac{5}{6}})$. In the whole space $\R\times \R$, Wang-Wang \cite{Wang-Wang} derived stability results with the threshold index bounded above by $(\nu^{\frac34}, \nu^{\frac{5}{4}})$. This threshold was later improved to $(\nu^{\frac13+}, \nu^{\frac23+})$ by Chen-Wang-Yang \cite{Chen-Wang-Yang}.

For additional stability results of various shear flows, the reader is referred to \cite{Bedrossian-Bianchini-Coti-Dolce,Coti-Del,Cui-Wang-Wang,Liang-Li-Zhai,Masmoudi-Said-Zhao,Ren-Wei,Zhai-Zhao,Zillinger,Zillinger1} and references therein.
For the study on the asymptotic stability of the Navier-Stokes equations, see \cite{Bedrossian-Germain-Masmoudi-1,Bedrossian-Germain-Masmoudi-2,Bedrossian-Germain-Masmoudi-3,Bedrossian-Germain-Masmoudi-4,Chen-Li-Wei-Zhang ,Chen-Wei-Zhang1, Chen-Wei-Zhang2, cl-jde2024,cl-non2025,Li-Shen-Zhang,Ionescu-Jia2,Ionescu-Jia1,Li-Wei-Zhang1, Masmoudi-Zhao,Wei-Zhang-2,Wei-Zhang-Zhao,Wei-Zhang-Zhao1,Jia1,Chen-Li-Miao} and references therein.

Both in nature and laboratory, laminar flows are rarely perfectly steady. In fact, they are often highly unsteady. A typical example is B\'{e}nard convection observed in laboratory. Consequently, investigating time-dependent shear flows holds considerable significance.

In this paper, we consider the time-dependent monotone shear flow $(U(t,y),0)$.
It is clear that for any $a,b\in\R$,
$$v_{U}=(U(t,y),0),\quad \vth_{a}=a,\quad P_{a,b}=ay+b,$$
is a solution of \eqref{Bou}, where the shear flows are governed by the heat equation
\begin{equation}\label{equ: U}
\left\{
\begin{aligned}
&\pa_{t}U=\nu\pa_{y}^{2}U,\\
&U(0,y)=U^{\mathrm{in}}(y),\quad U(t,0)=U^{\mathrm{in}}(0),\quad U(t,1)=U^{\mathrm{in}}(1).
\end{aligned}
\right.
\end{equation}
We study concave monotone shear flows with the initial data satisfying
\begin{equation}\label{equ:condition,shear,flow,M}
\left\{
\begin{aligned}
 &U^{\mathrm{in}}\in H^4(0,1),\quad \quad \pa_{y}U^{\mathrm{in}}\geq c_0>0,\\
 &\text{$\pa_{y}^{2}U^{\mathrm{in}}\geq 0$, or $\pa_{y}^{2}U^{\mathrm{in}}\leq 0$},
\quad \pa_{y}^{2}U^{\mathrm{in}}(0)=\pa_{y}^{2}U^{\mathrm{in}}(1)=0.
\end{aligned}
\right.
\end{equation}

Let $(u, \theta)=(v-v_{U}, \vth-\vth_{a})$ be the perturbation of the velocity and temperature. Under the non-slip boundary condition, it satisfies
\begin{equation}\label{pertu}
\left\{
\begin{aligned}
&\pat u-\nu \De u+u\cdot\na u+(\pa_yU u^{(2)},0)+U(t,y)\pa_xu+\na p=\theta e_{2},\\
&\pat \theta-\nu \De \theta+u\cdot\na \theta+U(t,y)\pa_x\theta=0,\\
&\na\cdot u=0,\\
&u(t,x,0)=u(t,x,1)=0,\quad \theta(t,x,0)=\theta(t,x,1)=0,\\
&u(0,x,y)=u^{\mathrm{in}}(x,y),\quad \theta(0,x,y)=\theta^{\mathrm{in}}(x,y).
\end{aligned}
\right.
\end{equation}
Let $\om=\pa_y u^{(1)}-\pa_x u^{(2)}$ be the vorticity and $\psi$ be the stream function with the velocity $u=(\pa_{y}\psi,-\pa_x \psi)$. The vorticity and temperature system can be written as
\begin{equation}\label{equ: omli11}
\left\{
\begin{aligned}
&\pat\om-\nu\De \om+u\cdot\na \om+U(t,y)\pa_x\om-\pa^2_{y}U(t,y)\psi=\pa_x \theta,\\
&\pat \theta-\nu \De \theta+u\cdot\na \theta+U(t,y)\pa_x\theta=0,\\
&\De\psi=\om,\\
&\psi(t,x,0)=\psi(t,x, 1)=\pa_y\psi(t,x,0)=\pa_y \psi(t,x,1),\quad \theta(t,x,0)=\theta(t,x,1)=0,\\
&\om(0,x,y)=\om^{\mathrm{in}}(x,y).
\end{aligned}
\right.
\end{equation}

Now, we state the main result as follows.

\begin{theorem}\label{Th: tran thre}
Let $(u,\theta)$ be the solution to \eqref{pertu}. There exist positive constants $\nu_0$, $\epsilon_{0}$, $c$, such that for $0<\nu\leq \nu_{0}$, if the initial perturbation satisfies
$$\|u^{\mathrm{in}}\|_{H^2}\leq c\nu^{\frac12}, \qquad \|\lan D_x\ran \theta^{\mathrm{in}}\|_{L^2}\leq c\nu^{\frac56},$$
then the solution $(u,\theta)$ satisfies the global stability estimates
\begin{align*}
\sum_{k\in\Z} E_k \leq C\nu^{\frac12}, \qquad \sum_{k\in\Z} G_k \leq C\nu^{\frac56},
\end{align*}
where the stability norms are given by
\begin{equation*}
\begin{aligned}
E_k=&\left\{
\begin{aligned}
&\|\om_0\|_{L^\infty L^2}, \quad k=0,\\
&|k|\| e^{\epsilon_0 \nu^{\frac13}t} u_k\|_{L^2L^2}+\|e^{\epsilon_0 \nu^{\frac13}t} u_k\|_{L^\infty L^\infty}\\
&+\|e^{\epsilon_0 \nu^{\frac13}t} \sqrt{1-(2y-1)^2} \om_k\|_{L^\infty L^2}+\nu^{\frac14}|k|^{\frac12}\|e^{\epsilon_0 \nu^{\frac13}t} \om_k\|_{L^2L^2}, \quad k\neq 0,
\end{aligned}
\right.\\
G_k=&\left\{
\begin{aligned}
&\|\theta_0\|_{L^\infty L^2}, \quad k=0,\\
&|k|^{\frac13}\|e^{\epsilon_0 \nu^{\frac13}t} \theta_k\|_{L^\infty L^2}+\nu^{\frac16}|k|^{\frac23}\|e^{\epsilon_0 \nu^{\frac13}t} \theta_k\|_{L^2L^2}, \quad k\neq 0,
\end{aligned}
\right.
\end{aligned}
\end{equation*}
with $f_k(t,y)=:\int_{\T} f(t,x,y) e^{-ikx}dx$ and $\|\cdot\|_{L^pL^q}=\|\cdot\|_{L^p_tL^q_y}$.
\end{theorem}

\begin{remark}
To our knowledge, Theorem \ref{Th: tran thre} is the first transition threshold result for the Boussinesq equation studying the monotone shear flow. Moreover, this result matches the known optimal transition threshold for two cases:  monotone shear flow governed by the Navier-Stokes equation in finite channel with non-slip boundary condition \cite{Li-Shen-Zhang} and Couette flow described by the Boussinesq equation in infinite channel with non-slip boundary condition \cite{Liang-Wu-Zhai}.
\end{remark}

\begin{remark}
Inspired by Wei-Zhang \cite{Wei-Zhang-2} and Liang-Wu-Zhai \cite{Liang-Wu-Zhai}, we bound the forcing term $\pa_x \theta$ on the right hand of the first equation by the estimate of $|D_x|^{\frac13}\theta$ rather than $\theta$, gaining an extra factor $\nu^{\frac{1}{12}}$ than the transition threshold in \cite{Masmoudi-Zhai-Zhao}.
\end{remark}

Compared with Couette flow,  the Orr-Sommerfeld equation associated with monotone flows contains a non-local term, which destroys the good structure of the linear operator. Moreover, the time dependent linear operator $-\nu\De+U(y,t)\pa_x+\De^{-1}\pa_x$ also poses additional challenges in deriving space-time estimates.
We follow the approach proposed in \cite{Li-Shen-Zhang}, which employs the resolvent estimates and the frozen-time technique to derive the desired space-time estimates:
\begin{align*}
&|k|^2\|e^{\epsilon \nu^{\frac13}t}u\|^2_{L^2 L^2}+\nu^{\frac12} |k|\|e^{\epsilon \nu^{\frac13}t}\om\|^2_{L^2 L^2}+\|e^{\epsilon \nu^{\frac13}t} \sqrt{1-(2y-1)^2}\om\|^2_{L^\infty L^2}+|k|\|e^{\epsilon \nu^{\frac13}t} u\|^2_{L^\infty L^\infty} \notag\\
& \lesssim \|\om^{\mathrm{in}}_{k}\|^2_{L^2}+\nu^{-1}\|e^{\epsilon \nu^{\frac13}t}(f_1,f_2)\|^2_{L^2 L^2} +\min\{\nu^{-\frac13}|k|^{-\frac23},\nu^{-1}|k|^{-2}\}\|e^{\epsilon \nu^{\frac13}t}f_3\|^2_{L^2 L^2},
\end{align*}
and
\begin{align*}
&\|e^{\epsilon \nu^{\frac13}t}\theta\|^2_{L^\infty L^2}+\max\{\nu^{\frac13}|k|^{\frac23},\nu |k|^2\}\|e^{\epsilon \nu^{\frac13}t}\theta\|^2_{L^2 L^2}\\
& \lesssim\|\theta^{\mathrm{in}}_{k}\|^2_{L^2}+\min\{\nu^{-\frac13}|k|^{\frac43},\nu^{-1}\}\|e^{\epsilon \nu^{\frac13}t}g_1\|^2_{L^2 L^2}+\nu^{-1}\|e^{\epsilon \nu^{\frac13}t}g_2\|^2_{L^2 L^2}.
\end{align*}

Based on these space-time estimates, we finally derive the transition threshold by using the divergence formulation for the nonlinear term in the vorticity equation and a delicate analysis about the interactions of high-low frequencies between $u$ with $\om$ and $\theta$.

\medskip

\noindent{\bf Notations.}
Throughout this paper, $C$ denotes a general constant independent of $\nu, k, \la$, and it may vary from line to line. The notation $A\lesssim B$ means $A\leq CB$.
In the following, we will omit the subscript $k$ for $\om, \theta, \psi, \phi, F$ while still keep the dependence on $k$ in the actual estimates.

\section{Space-Time estimates of the linearized Boussinesq equations}\label{sec: space-time estimate}
In this section, we consider the linearized Boussinesq system
\begin{equation}\label{equ: om}
\left\{\begin{aligned}
&\pa_t\om-\nu(\pa^2_y-k^2)\om+ik U\om-ik \pa^{2}_{y}U\psi=-i k f_1-\pa_y f_2-f_3,\\
&\pat \theta-\nu (\pa^2_y-k^2) \theta+ikU\theta=-ik g_1-\pa_y g_2,\\
& (\pa^2_y-\al^2)\psi=\om,\\
&\psi(t,0)=\psi(t,1)=\pa_y \psi( t, 0)=\pa_y \psi( t, 1)=0,\quad \theta(t,0)=\theta(t,1)=0,\\
&\om(0,y)=\om_{k}^{\mathrm{in}}(y), \quad \theta(0,y)=\theta_{k}^{\mathrm{in}}(y).
\end{aligned}
\right.
\end{equation}

To obtain the space-time estimates of \eqref{equ: om}, we primarily follow the method developed in \cite{Chen-Li-Wei-Zhang,Li-Shen-Zhang}. Decompose the equation \eqref{equ: om} into its inhomogeneous and homogeneous components. The space-time estimates for these two components are  deduced by the resolvent estimates established in \cite{Li-Shen-Zhang}.

We first consider the Orr-Sommerfeld equation with the time-independent flow $V(y)$:
\begin{equation}\label{equ: om-freezed}
\left\{\begin{aligned}
&-\nu(\pa^2_y-k^2)w +ik(V-\la)w -ik V''\phi+o(\nu,k)w= F,\\
& (\pa^2_y-\al^2)\phi=w,\\
&\phi(t,0)=\phi(t,1)=\phi'( t, 0)=\phi'( t, 1)=0,\\
&w(0,y)=w^{\mathrm{in}}(y),
\end{aligned}
\right.
\end{equation}
where $\la\in \R$, $|o(\nu,k)|\ll (\nu k^{2})^{\frac{1}{3}}$ and $V(y)$ satisfies
\begin{align}\label{cond V}
V(y)\in H^4(0,1), \quad \inf_{y\in(0,1)}V'\geq c_{0}>0,\,\,\,\quad  V'' >0 \quad \text{or}\quad V''<0.
\end{align}

Decompose $w$ into
\begin{align}\label{equ:decomposition,wna,w1,w2}
w=w_{Na}+c_{1}w_{1}+c_{2}w_{2},
\end{align}
where
\begin{equation}\label{def: c1c2}
\begin{aligned}
c_{1}=\int^1_{0}\frac{\sinh  k(1-y)}{\sinh k}w_{Na}(y)dy,\qquad c_{2}=\int^1_{0}\frac{\sinh (ky)}{\sinh k}w_{Na}(y)dy,
\end{aligned}
\end{equation}
$w_{Na}$ solves the inhomogeneous OS equation with the Navier-slip boundary condition
\begin{equation}\label{equ:Navier-slip,OS}
\left\{
\begin{aligned}
&-\nu(\pa^2_y-k^2)w_{Na} +ik(V-\la)w_{Na} -ik V''\phi_{Na}+o(\nu,k)w_{Na} =F ,\\
&\phi_{Na}=(\pa^2_y-k^2)^{-1}w_{Na}, \quad \phi_{Na}(0)=\phi_{Na}(1)=w_{Na}(0)=w_{Na}(1)=0,\\
\end{aligned}
\right.
\end{equation}
and $w_1$, $w_2$ solve the following homogeneous OS equations
\begin{equation}\label{equ: psi1}
\left\{
\begin{aligned}
&-\nu(\pa^2_y-k^2)w_{1}+i k(V-\la)w_{1}-ik V''\phi_{1}+o(\nu,k)w_{1}=0,\\
&\phi_{1}=(\pa^2_y-k^2)^{-1}w_{1},\quad \phi_{1}(0)=\phi_{1}(1)=0,\quad \pa_y\phi_{1}(0)=0, \quad \pa_y\phi_{1}( 1)=1,
\end{aligned}
\right.
\end{equation}
and
\begin{equation}\label{equ: psi2}
\left\{
\begin{aligned}
&-\nu(\pa^2_y-k^2)w_{2}+i k(V-\la)w_{2}-ik V''\phi_{2}+o(\nu,k)w_{2}=0,\\
&\phi_{2}=(\pa^2_y-k^2)^{-1}w_{2},\quad \phi_{2}(0)=\phi_{2}(1)=0,\quad \pa_y\phi_{2}(0)=1,\quad \pa_y\phi_{2}(1)=0.
\end{aligned}
\right.
\end{equation}

The following estimates are derived from \cite{Li-Shen-Zhang,Chen-Li-Wei-Zhang}:
\begin{itemize}
\item Resolvent estimates with the Navier-slip boundary condition:
 \begin{align}
 \nu^{\frac{1}{6}}|k|^{\frac{4}{3}}\|u_{Na}\|_{L^2}+\nu^{\frac{1}{3}} k^{\frac{2}{3}}\|w_{Na}\|_{L^2}+\nu^{\frac23}|k|^{\frac13}\|(\pa_y,k)w_{Na}\|_{L^2}\lesssim&  \|F\|_{L^2},\label{esti: wL2}\\
\nu^{\frac{1}{6}}|k|^{\frac43}\|w_{Na}\|_{L^2}+\nu^{\frac{1}{12}}|k|^{\frac{5}{3}}\|w_{Na}\|_{L^1}+ |k|^2\|(V(y)-\la)w_{Na}\|_{L^2} \lesssim & \|F\|_{H_{k}^{1}}, \label{esti: wH1}\\
\nu^{\frac{1}{2}}  |k|\left\|u_{Na}\right\|_{L^2}+\nu^{\frac{2}{3}} |k|^{\frac{1}{3}}\|w_{Na}\|_{L^2}+\nu\|(\pa_y,k) w_{Na}\|_{L^2}
 \lesssim & \|F\|_{H_{k}^{-1}}, \label{esti: wH-1}
\end{align}
where the velocity $u_{Na}=(\pa_y\phi_{Na},-ik\phi_{Na})$, and the norms are given by
\begin{align*}
\|F\|_{H_{k}^{1}}=\|(\pa_y,k) F\|_{L^2},\quad \|F\|_{H_{k}^{-1}}=\sup \{|\lan F,\phi\ran|: \phi\in H^1_0([0,1]), \, \|\phi\|_{H_{k}^{1}}=1\}.
\end{align*}

\item Estimates for the boundary layer correctors:
\begin{align}
&(1+|k( \la-V(0))|)^{\frac34}|c_1|\|\rho^{\frac12}_{k} w_{1}\|_{L^2}+(1+|k( \la-V(1))|)^{\frac34}|c_2|\|\rho^{\frac12}_{k} w_{2}\|_{L^2}\label{c1rho+FL2}\\
&+\nu^{\frac18}|k|^{\frac14}(1+|k( \la-V(0))|)^{\frac38}|c_1|\|\rho^{-\frac14}_{k} w_{1}\|_{L^2}\notag\\
&+\nu^{\frac18}|k|^{\frac14}(1+|k( \la-V(1))|)^{\frac38}|c_2|\|\rho^{-\frac14}_{k} w_{2}\|_{L^2}
 \lesssim \nu^{-\frac13}|k|^{-\frac23}\|F\|_{L^2}, \notag\\
&(1+|k(\la-V(0))|)^{\frac34}|c_1|\|\rho^{\frac12}_k w_{1}\|_{L^2}+(1+|k(\la-V(1))|)^{\frac34}|c_2|\|\rho^{\frac12}_k w_{2}\|_{L^2} \label{c1rho+H-1}\\
&+\nu^{\frac18}|k|^{\frac14}(1+|k(\la-V(0))|)^{\frac38}|c_1|\|\rho^{-\frac14}_k w_{1}\|_{L^2}\notag \\
&+\nu^{\frac18}|k|^{\frac14}(1+|k(\la-V(1))|)^{\frac38}|c_2|\|\rho^{-\frac14}_k w_{2}\|_{L^2}
\lesssim \nu^{-\frac{2}{3}}|k|^{-\frac{1}{3}}\|F\|_{H_{k}^{-1}},\notag\\
&(1+|k( \la-V(0))|)^{\frac34}|c_1|\|w_{1}\|_{L^2}+(1+|k( \la-V(1))|)^{\frac34}|c_2|\|w_{2}\|_{L^2}\lesssim \nu^{-\frac{1}{3}}|k|^{-\frac53}\|F\|_{H_{k}^{1}},\label{c1 FH1}
\end{align}
and
\begin{align}\label{esti:rho}
\int_{\R}(|\la- kV(0)|^{\frac32}\rho_k(y)+(\nu k^2)^{\frac14}|\la-kV(0)|^{\frac34}\rho^{-\frac12}_k(y))^{-1} d\la
=C \nu^{-\frac16}| k|^{-\frac13},
\end{align}
where
\begin{equation}\label{equ:def,rho,k}
\rho_{k}(y)=\left\{
\begin{aligned}
&\nu^{-\frac13}|k|^{\frac13}y, \qquad\qquad\, y\in [0,\nu^{\frac13}|k|^{-\frac13}], \\
&1, \qquad\qquad\qquad\quad\,\, y\in [\nu^{\frac13}|k|^{-\frac13}, 1-\nu^{\frac13}|k|^{-\frac13}],\\
&\nu^{-\frac13}|k|^{\frac13}(1-y), \qquad y\in [1-\nu^{\frac13}|k|^{-\frac13}, 1].
\end{aligned}
\right.
\end{equation}
\end{itemize}

Let $(\om,\psi)$ be the solution to
\begin{equation}\label{equ: om const}
\left\{\begin{aligned}
&\pa_t \om -\nu(\pa^2_y-k^2)\om +ikV\om -ik V''\psi=-ik f_1-\pa_y f_2-f_3-f_4,\\
& (\pa^2_y-\al^2)\psi=\om,\\
&\psi(t,0)=\psi(t,1)=\psi'( t, 0)=\psi'( t, 1)=0,\\
&\om(0,y)=\om_{k}^{\mathrm{in}}(y).
\end{aligned}
\right.
\end{equation}
\begin{lemma}\label{lemma: space-time froze}
Let $\nu k^2\leq \|V'\|_{L^\infty}$ and $(\om,\psi)$ be the solution to \eqref{equ: om const}. Then there exist constants $\nu_0>0$ and $\epsilon_0>0$ such that, for $\nu\in(0,\nu_0]$, $\epsilon\in[0,\epsilon_{0}]$, we have
\begin{align*}
&|k|^2\|e^{\epsilon \nu^{\frac13}t}u\|^2_{L^2 L^2}+\nu^{\frac12} |k|\|e^{\epsilon \nu^{\frac13}t}\om\|^2_{L^2 L^2}+\|e^{\epsilon \nu^{\frac13}t} \sqrt{1-(2y-1)^2}\om\|^2_{L^\infty L^2}+|k|\|e^{\epsilon \nu^{\frac13}t} u\|^2_{L^\infty L^\infty} \\
\lesssim& E^{\mathrm{in}}+ \nu^{-1}\|e^{\epsilon \nu^{\frac13} t}(f_1, f_2)\|^2_{L^2 L^2}+\nu^{-\frac13}|k|^{-\frac23}\|e^{\epsilon \nu^{\frac13} t}f_3\|^2_{L^2 L^2}+\nu^{-\frac{1}{6}}|k|^{-\frac73}\|e^{\epsilon \nu^{\frac13} t}f_4\|^2_{L^2 H_{k}^{1}},
\end{align*}
where $E^{\mathrm{in}}=:|k|^{-2} \|\pa_y\om^{\mathrm{in}}\|^2_{L^2}+\|u^{\mathrm{in}}\|^2_{H^1}$.
\end{lemma}
\begin{proof}
Using the method employed in the proofs of \cite[Propositions 6.1 and 6.7]{Chen-Li-Wei-Zhang}, we are left to prove
\begin{align}\label{esti: omconst}
&\nu^{\frac12}|k|\|e^{\epsilon \nu^{\frac13}t}\om\|^2_{L^\infty L^2}+|k|^2\bn{e^{\epsilon \nu^{\frac13}t} u}^2_{L^2L^2}+\nu^{\frac12} | k| \|e^{\epsilon \nu^{\frac13}t}\om\|^2_{L^2L^2}+\nu^{\frac13}|k|^{\frac23}\|e^{\epsilon \nu^{\frac13}t}\rho^{\frac12}_k\om\|^2_{L^2L^2}\\
\lesssim& E^{\mathrm{in}}+ \nu^{-1}\|e^{\epsilon \nu^{\frac13} t}(f_1, f_2)\|^2_{L^2 L^2}+\nu^{-\frac13}|k|^{-\frac23}\|e^{\epsilon \nu^{\frac13} t}f_3\|^2_{L^2 L^2}+\nu^{-\frac{1}{6}}|k|^{-\frac73}\|e^{\epsilon \nu^{\frac13} t}f_4\|^2_{L^2 H_{k}^{1}}.\notag
\end{align}

The estimates for the last three energy terms have been established in \cite{Li-Shen-Zhang} that
\begin{align*}
&|k|^2\bn{e^{\epsilon \nu^{\frac13}t} u}^2_{L^2L^2}+\nu^{\frac12} | k| \|e^{\epsilon \nu^{\frac13}t}\om\|^2_{L^2L^2}+\nu^{\frac13}|k|^{\frac23}\|e^{\epsilon \nu^{\frac13}t}\rho^{\frac12}_k\om\|^2_{L^2L^2}\\
\lesssim&  \nu^{-1}\|e^{\epsilon \nu^{\frac13} t}(f_1, f_2)\|^2_{L^2 L^2}+\nu^{-\frac13}|k|^{-\frac23}\|e^{\epsilon \nu^{\frac13} t}f_3\|^2_{L^2 L^2}+\nu^{-\frac{1}{6}}|k|^{-\frac73}\|e^{\epsilon \nu^{\frac13} t}f_4\|^2_{L^2 H_{k}^{1}}.
\end{align*}

Now, we focus on the bound of $\nu^{\frac12}|k|\|e^{\epsilon \nu^{\frac13}t}\om\|^2_{L^\infty L^2}$.
Give the decomposition
$$\om=\om_{I}+\om_{H},$$
where the inhomogeneous part $\om_{I}$ satisfies
\begin{equation}\label{equ: omI}
\left\{
\begin{aligned}
&\pa_t \om_{I}-\nu(\pa^2_y- k^2)\om_{I}+i kV\om_{I}-ik \pa^2_y V\psi_{I}=-i k f_1-\pa_y f_2-f_3-f_4, \\
&\om_{I}=(\pa^2_y- k^2)\psi_{I},\quad \psi_{I}(0)=\psi_{I}(1)=\psi'_{I}(0)=\psi'_{I}(1)=0,\\
&\om_{I}(0,y)=0,
\end{aligned}
\right.
\end{equation}
and the homogeneous part $\om_{H}$ satisfies
\begin{equation}\label{equ: omH}
\left\{
\begin{aligned}
&\pa_t \om_{H}-\nu(\pa^2_y- k^2)\om_{H}+i kV\om_{H}-ik \pa^2_y V\psi_{H}=0,\\
&\om_{H}=(\pa^2_y- k^2)\psi_{H},\quad \psi_{H}(0)=\psi_{H}(1)=\psi'_{H}(0)=\psi'_{H}(1)=0,\\
&\om_{H}(0,y)=\om_{k}^{\mathrm{in}}(y).
\end{aligned}
\right.
\end{equation}

Let
\begin{equation}\label{equ: omNa1}
\left\{
\begin{aligned}
&\pa_t \om^{(j)}_{Na}-\nu (\pa^2_y-k^2)\om^{(j)}_{Na}+ikV\om^{(j)}_{Na}-ikV''\psi^{(j)}_{Na}=-h_{j},\quad j\in \{1,2,3,4\},\\
&\om^{(j)}_{Na}(t,k,0)=\om^{(j)}_{Na}(t,k,1)=0, \quad \psi^{(j)}_{Na}(t,k,0)=\psi^{(j)}_{Na}(t,k,1)=0,
\end{aligned}
\right.
\end{equation}
with
$$h_{1}=ikf_1,\qquad h_{2}=\pa_y f_2,\qquad h_{3}=f_3,\qquad h_{4}=f_4.$$
We further decompose $\om_{I}$ into
\begin{align*}
\om_{I}=\sum_{1\leq j\leq 4}\om^{(j)}_{Na}+\sum_{1\leq j\leq 4}\sum_{1\leq l\leq 2}\om^{(j)}_{l},
\end{align*}
where for $t>0$, $j\in\{1,2,3,4\}$, $l\in \{1,2\}$,
\begin{align*}
&e^{\epsilon \nu^{\frac13} t}\om^{(j)}_{Na}(t, k,y) =\frac{1}{2\pi}\int_{\R}w^{(j)}_{Na}(\la, k,y) e^{it\la} d\la, \\
&e^{\epsilon \nu^{\frac13} t}\om^{(j)}_{l}(t, k,y)=\frac{1}{2\pi}\int_{0}^{\infty} c^{(j)}_{1}w^{(j)}_{l}(\la, k,y)e^{it\la} d\la, \\
& c^{(j)}_{1}=\int^1_{0}\frac{\sinh k(1-y)}{\sinh k}w^{(j)}_{Na}(y)dy,\qquad c^{(j)}_{2}=\int^1_{0}\frac{\sinh (ky)}{\sinh k}w^{(j)}_{Na}(y)dy,
\end{align*}
with
\begin{align*}
\left\{
\begin{aligned}
& -\nu (\pa^2_y-k^2)w^{(j)}_{1}+ik (V-\la/k)w^{(j)}_{1}-ikV''\phi^{(j)}_{1}-\epsilon\nu^{\frac13}w^{(j)}_{1}=0,\\
&\phi^{(j)}_{1}(t,k,0)=\phi^{(j)}_{1}(t,k,1)=0, \quad \pa_y \phi^{(j)}_{1}(t,k,0)=0, \quad \pa_y\phi^{(j)}_{1}(t,k,1)=1.
\end{aligned}
\right.
\end{align*}
and
\begin{align*}
\left\{
\begin{aligned}
& -\nu (\pa^2_y-k^2)w^{(j)}_{2}+ik (V-\la/k)w^{(j)}_{2}-ikV''\phi^{(j)}_{2}-\epsilon\nu^{\frac13}w^{(j)}_{2}=0,\\
&\phi^{(j)}_{2}(t,k,0)=\phi^{(j)}_{2}(t,k,1)=0, \quad \pa_y \phi^{(j)}_{2}(t,k,0)=1, \quad \pa_y\phi^{(j)}_{2}(t,k,1)=0.
\end{aligned}
\right.
\end{align*}

We use \eqref{esti: wL2} and Plancherel's theorem to get
\begin{align}
&\nu^{\frac13}|k|^{\frac23}\|e^{\epsilon \nu^{\frac13}t}\om^{(1)}_{Na}\|^2_{L^2L^2}+|k|^{2}\|e^{\epsilon \nu^{\frac13}t}u^{(1)}_{Na}\|^2_{L^2L^2}\label{omNa1L2L2}\\
\sim & \nu^{\frac13}|k|^{\frac23}\|w^{(1)}_{Na}\|^2_{L^2L^2}+|k|^2 \|(\pa_y, k)\phi^{(1)}_{Na}\|^2_{L^2L^2} \notag\\
\lesssim &\nu^{-\frac13}|k|^{\frac43}\|e^{\epsilon \nu^{\frac13}t}f_1\|^2_{L^2L^2}\lesssim \nu^{-1}\|e^{\epsilon \nu^{\frac13}t}f_1\|^2_{L^2L^2}.\notag
\end{align}
Testing \eqref{equ: omNa1} by $\om^{(1)}_{Na}$ and taking the real part, we obtain
\begin{align*}
\frac12 \frac{d}{dt}\|\om^{(1)}_{Na}\|^2_{L^2}+\nu \|(\pa_y, k)\om^{(1)}_{Na}\|^2_{L^2}
\leq |\lan k V'' \psi^{(1)}_{Na}, \pa_y \psi^{(1)}_{Na} \ran|+\nu^{-1}\|e^{\epsilon \nu^{\frac13} t}f_1\|^2_{ L^2}+\frac12\nu \| k\om^{(1)}_{Na}\|^2_{L^2},
\end{align*}
which, together with \eqref{omNa1L2L2}, implies
\begin{align}\label{omNa1LinftyL2}
&\|e^{\epsilon \nu^{\frac13}t}\om^{(1)}_{Na}\|^2_{L^\infty L^2}+\nu \|e^{\epsilon \nu^{\frac13}t}(\pa_y, k)\om^{(1)}_{Na}\|^2_{L^2L^2}\\
\lesssim & \nu^{\frac13}|k|^{\frac23}\|e^{\epsilon \nu^{\frac13}t}(\pa_y, k)\om^{(1)}_{Na}\|^2_{L^2L^2}+\|e^{\epsilon \nu^{\frac13}t}u^{(1)}_{Na}\|^2_{L^2L^2}+\nu^{-1}\|e^{\epsilon \nu^{\frac13} t}f_1\|^2_{L^2 L^2} \notag\\
\lesssim &\nu^{-1}\|e^{\epsilon \nu^{\frac13} t}f_1\|^2_{L^2 L^2}.\notag
\end{align}

Similarly, we also have
\begin{align}\label{omNaf1-4L2}
&\|e^{\epsilon \nu^{\frac13}t}\om^{(2)}_{Na}\|^2_{L^\infty L^2}+\|e^{\epsilon \nu^{\frac13}t}\om^{(3)}_{Na}\|^2_{L^\infty L^2}+\|e^{\epsilon \nu^{\frac13}t}\om^{(4)}_{Na}\|^2_{L^\infty L^2}\\
\lesssim & \nu^{-1}\|e^{\epsilon \nu^{\frac13} t} f_2\|^2_{L^2 L^2}+\nu^{-\frac13}|k|^{-\frac23}\|e^{\epsilon \nu^{\frac13} t}f_3\|^2_{L^2 L^2}+\nu^{-\frac{1}{6}}|k|^{-\frac73}\|e^{\epsilon \nu^{\frac13} t}f_4\|^2_{L^2 H_{k}^{1}}.\notag
\end{align}

For $\om^{(1)}_{1}$, a direct calculation gives
\begin{equation}\label{om21}
\begin{aligned}
&|e^{\epsilon \nu^{\frac13}t}\om^{(1)}_{1}(t, k,y)|^2\leq \Big|\int_{\R}|c^{(1)}_{1}w^{(1)}_{1}(\la, k,y)| d\la\Big|^2\\
\leq& \int_{\R}|\la- kV(0)|^{\frac32}\rho_k(y)|c^{(1)}_{1}w^{(1)}_{1}|^2+\nu^{\frac14}| k|^{\frac12}|\la- kV(0)|^{\frac34}\rho^{-\frac12}_k (y)|c^{(1)}_{1}w^{(1)}_{1}|^2 d\la\\
&\times \int_{\R} \big(|\la- kV(0)|^{\frac32}\rho_k(y)+(\nu k^2)^{\frac14}|\la- kV(0)|^{\frac34}\rho^{-\frac12}_k(y)\big)^{-1} d\la.
\end{aligned}
\end{equation}
In terms of \eqref{c1rho+FL2} and \eqref{esti:rho}, we obtain
\begin{align*}
&\|e^{\epsilon \nu^{\frac13}t}\om^{(1)}_{1}(t)\|^2_{L^2_y}\\
\lesssim &\nu^{-\frac16}| k|^{-\frac13}\int_{\R}\Big(|\la-kV(0)|^{\frac32}\|\rho^{\frac12}_kw^{(1)}_{1}\|^2_{L^2_y}+\nu^{\frac14}| k|^{\frac12}|\la-kV(0)|^{\frac34}\|\rho^{-\frac14}_kw^{(1)}_{1}\|^2_{L^2_y} \Big)|c^{(1)}_{1}|^2 d\la \\
\lesssim & \nu^{-\frac16}| k|^{-\frac13}\int_{\R}\nu^{-\frac23}|k|^{\frac23}\|F_1(\la)\|^2_{L^2_y} d\la\lesssim \nu^{-\frac56}| k|^{\frac{1}{3}}\| e^{\epsilon \nu^{\frac13}t} f_1\|^2_{L^2 L^2},
\end{align*}
which implies
\begin{align}\label{om21L2f2L2}
\nu^{\frac12} | k|\|e^{\epsilon \nu^{\frac13}t}\om^{(1)}_{1}\|^2_{L^\infty L^2}\lesssim  \nu^{-1} \|e^{\epsilon \nu^{\frac13}t}f_1\|^2_{L^2L^2}.
\end{align}

Similarly, we can use \eqref{c1rho+FL2} and \eqref{c1rho+H-1} to derive
\begin{align}
\nu^{\frac12} | k|\|e^{\epsilon \nu^{\frac13}t}\om^{(1)}_{2}\|^2_{L^\infty L^2}\lesssim & \nu^{-1} \|e^{\epsilon \nu^{\frac13}t}f_1\|^2_{L^2L^2}, \label{om12L2f1L2}\\
\nu^{\frac12} | k|\|e^{\epsilon \nu^{\frac13}t}\om^{(2)}_{1}\|^2_{L^\infty L^2}+\nu^{\frac12} | k|\|e^{\epsilon \nu^{\frac13}t}\om^{(2)}_{2}\|^2_{L^\infty L^2}\lesssim &\nu^{-1}\|e^{\epsilon \nu^{\frac13}t}f_2\|^2_{L^2L^2},\label{om2122f2L2} \\
\nu^{\frac12} | k|\|e^{\epsilon \nu^{\frac13}t}\om^{(3)}_{1}\|^2_{L^\infty L^2}+\nu^{\frac12} | k|\|e^{\epsilon \nu^{\frac13}t}\om^{(3)}_{2}\|^2_{L^\infty L^2}\lesssim &\nu^{-\frac13} | k|^{-\frac23}\|e^{\epsilon \nu^{\frac13}t}f_3\|^2_{L^2L^2},\label{om3132f3L2}
\end{align}

For $\om^{(4)}_{1}$, using \eqref{c1 FH1}, we get
\begin{align*}
\|e^{\epsilon \nu^{\frac13}t}\om^{(4)}_{1}\|_{L^2}\leq& \frac{1}{2\pi}\int_{\R}|c^{(1)}_{1}|\|w^{(4)}_{1}\|_{L^2}d\la\\
\lesssim & \|(1+|\la-kV(0)|)^{-\frac34}\nu^{-\frac{1}{3}}|k|^{-\frac76} \|F_4\|_{H^1_{k}}\|_{L^1}\\
\lesssim & \nu^{-\frac13}|k|^{-\frac76}\|(1+|\la-kV(0)|)^{-\frac34}\|_{L^2}\|F_4\|_{L^2H^1_k}\leq \nu^{-\frac13}|k|^{-\frac76}\|e^{\epsilon \nu^{\frac13}t}f_4\|_{L^2H^1_k},
\end{align*}
which implies
\begin{align}\label{om41L2f4L2}
\nu^{\frac12} | k|\|e^{\epsilon \nu^{\frac13}t}\om^{(4)}_{1}\|^2_{L^\infty L^2}\leq \nu^{-\frac16}|k|^{-\frac43}\|e^{\epsilon \nu^{\frac13}t}f_4\|^2_{L^2H^1_k}.
\end{align}
Similarly, we also have
\begin{align}\label{om43L2f4L2}
\nu^{\frac12} | k|\|e^{\epsilon \nu^{\frac13}t}\om^{(4)}_{2}\|^2_{L^\infty L^2}\leq \nu^{-\frac16}|k|^{-\frac43}\|e^{\epsilon \nu^{\frac13}t}f_4\|^2_{L^2H^1_k}.
\end{align}

Combining  \eqref{omNa1LinftyL2}, \eqref{omNaf1-4L2}, and \eqref{om21L2f2L2}--\eqref{om43L2f4L2}, we derive
\begin{align*}
&\nu^{\frac12} |k|\|e^{\epsilon \nu^{\frac13}t}\om_{I}\|^2_{L^\infty L^2}\\
\lesssim & \nu^{-1}\|e^{\epsilon \nu^{\frac13} t}(f_1, f_2)\|^2_{L^2 L^2}+\nu^{-\frac13}|k|^{-\frac23}\|e^{\epsilon \nu^{\frac13} t}f_3\|^2_{L^2 L^2}+\nu^{-\frac{1}{6}}|k|^{-\frac73}\|e^{\epsilon \nu^{\frac13} t}f_4\|^2_{L^2 H_{k}^{1}}.
\end{align*}

The proof for the estimate $\|e^{\epsilon \nu^{\frac13}t}\om_{H}\|_{L^\infty L^2}$ is similar to the process as in \cite[Lemma 3.5]{Li-Shen-Zhang}. Hence, we complete the proof of Lemma \ref{lemma: space-time froze}.
\end{proof}

A similar proof as in \cite{Masmoudi-Zhai-Zhao} gives the following result.
\begin{lemma}\label{space-time thete}
Let $\theta$ be the solution of
\begin{equation}
\left\{
\begin{aligned}
&\pat \theta-\nu (\pa^2_y-k^2) \theta+ikV\theta=-ik g_1-\pa_y g_2,\\
&\theta(t,0)=\theta(t,1)=0,\\
&\theta(0,y)=\theta_{k}^{\mathrm{in}}(y).
\end{aligned}
\right.
\end{equation}
Then, there exists $\epsilon_0\in(0,1)$ such that for $\epsilon\in(0,\epsilon_0)$, we have
\begin{align*}
&\|e^{\epsilon_0 \nu^{\frac13}t}\theta\|^2_{L^\infty L^2}+\max\{\nu^{\frac13}|k|^{\frac23},\nu |k|^2\}\|e^{\epsilon_0 \nu^{\frac13}t}\theta\|^2_{L^2L^2}\\
\leq & \|\theta^{\mathrm{in}}\|^2_{L^2}+\min\{\nu^{-\frac13}|k|^{\frac43},\nu^{-1}\}\|e^{\epsilon_0 \nu^{\frac13}t}g_{1}\|^2_{L^2L^2}+\nu^{-1}\|e^{\epsilon_0 \nu^{\frac13}t}g_{2}\|^2_{L^2L^2}.
\end{align*}
\end{lemma}
With Lemmas \ref{lemma: space-time froze} and \ref{space-time thete} at hand, we can obtain the following proposition.
 \begin{proposition}
Let $(\om,\theta)$ be the solution to \eqref{equ: om}. Then
there exist constants $\nu_0>0$ and $\epsilon_0>0$ such that, for $\nu\in(0,\nu_0]$, $\epsilon\in[0,\epsilon_{0})$, it holds that
\begin{align}\label{equ:low frequency}
&|k|^2\|e^{\epsilon \nu^{\frac13}t}u\|^2_{L^2 L^2}+\nu^{\frac12} |k|\|e^{\epsilon \nu^{\frac13}t}\om\|^2_{L^2 L^2}\\
&+\|e^{\epsilon \nu^{\frac13}t} \sqrt{1-(2y-1)^2}\om\|^2_{L^\infty L^2}+|k|\|e^{\epsilon \nu^{\frac13}t} u\|^2_{L^\infty L^\infty} \notag\\
\lesssim &\|\om^{\mathrm{in}}_{k}\|^2_{L^2}+\nu^{-1}\|e^{\epsilon \nu^{\frac13}t}(f_1,f_2)\|^2_{L^2 L^2} +\min\{\nu^{-\frac13}|k|^{-\frac23},\nu^{-1}|k|^{-2}\}\|e^{\epsilon \nu^{\frac13}t}f_3\|^2_{L^2 L^2},\notag
\end{align}
and
\begin{align}\label{esti: theta}
&\|e^{\epsilon \nu^{\frac13}t}\theta\|^2_{L^\infty L^2}+\max\{\nu^{\frac13}|k|^{\frac23},\nu |k|^2\}\|e^{\epsilon \nu^{\frac13}t}\theta\|^2_{L^2 L^2}\\
\lesssim &\|\theta^{\mathrm{in}}_{k}\|^2_{L^2}+\min\{\nu^{-\frac13}|k|^{\frac43},\nu^{-1}\}\|e^{\epsilon \nu^{\frac13}t}g_1\|^2_{L^2 L^2}+\nu^{-1}\|e^{\epsilon \nu^{\frac13}t}g_2\|^2_{L^2 L^2}.\notag
\end{align}
\end{proposition}
\begin{proof}
We can adopt the same proof process as \cite[Propositions 3.2 and 3.7]{Li-Shen-Zhang} to conclude \eqref{equ:low frequency}. For the sake of completeness, we present the proof of \eqref{esti: theta}.

For $\nu k^2\geq 1$, we directly test the second equation of \eqref{equ: om} with $\theta$ to obtain
\begin{align*}
\frac12\frac{d}{dt}\|\theta\|^2_{L^2}+\nu\|(\pa_y,k)\theta\|^2_{L^2}
=&\operatorname{Re}(\lan i k g_1,\theta\ran-\lan f_2,\pa_y \theta\ran)\\
\leq &\frac12\nu\|(\pa_y,k)\theta\|^2_{L^2}+C\nu^{-1}\lrs{\|g_1\|^2_{L^2}+\|g_2\|^2_{L^2}}.
\end{align*}
Thanks to
$$\frac{d}{dt}\|e^{\epsilon \nu^{\frac13}t} \theta\|^2_{L^2}=e^{\epsilon \nu^{\frac13}t}\frac{d}{dt}\| \theta\|^2_{L^2}+\epsilon \nu^{\frac13}\|e^{\epsilon \nu^{\frac13}t} \theta\|^2_{L^2},$$
we have
\begin{align*}
\frac{d}{dt}\|e^{\epsilon \nu^{\frac13}t} \theta\|^2_{L^2}+\nu\|e^{\epsilon \nu^{\frac13}t}(\pa_y,k)\theta\|^2_{L^2}\leq C\nu^{-1}\|e^{\epsilon \nu^{\frac13}t}(g_1,g_2)\|^2_{L^2}+
\epsilon \nu^{\frac13}\|e^{\epsilon \nu^{\frac13}t} \theta\|^2_{L^2},
\end{align*}
which implies
\begin{align}\label{nk2geq1hteta}
\|e^{\epsilon \nu^{\frac13}t} \theta\|^2_{L^\infty L^2}+\nu  \|e^{\epsilon \nu^{\frac13}t}(\pa_y,k)\theta\|^2_{L^2L^2}\lesssim \| \theta^{\mathrm{in}}\|^2_{L^2}+\nu^{-1}\|e^{\epsilon \nu^{\frac13}t}(g_1, g_2)\|^2_{L^2L^2}.
\end{align}

For $\nu k^2\leq 1$, let
$$ t_{j}=j\nu^{-\frac{1}{3}}.$$

Without loss of generality, we assume $T=t_{N+1}$ and decompose the solution
$\theta$ to the second equation of \eqref{equ: om} into $\theta=\sum_{0\leq j\leq N} \theta_{j}$, where  the component $\theta_{j}(j\geq 0)$ satisfies
\begin{equation}
\left\{
\begin{aligned}
&\pa_{t}\theta_{j}-\nu (\pa^2_y-k^2)\theta_{j}+ik U(t_{j+1},y)\theta_{j}=[ik g_{1}+\pa_{y} g_2+G_{j}]\chi_{I_{j}},\\
&\theta_{j}(t,k,0)=\theta_{j}(t,k,1)=0,\\
&\theta_{j}(t_{j},y)=1_{\lr{j=0}}\theta^{\mathrm{in}}(y),
\end{aligned}
\right.
\end{equation}
with $I_{j}=:[t_{j}, t_{j+1})$ and
$G_{j}:=\sum_{j'=0}^{j}ik (U(t_{j'+1},y)-U(t,y))\theta_{j'}.$

Let
$$\|f\|_{Y(a,b)}:=\nu^{\frac16}|k|^{\frac13}\|f\|_{L^2(a, b;L^2)}, \qquad \|f\|_{X(a,b)}:=\|f\|_{L^\infty(a,b; L^2)}+\|f\|_{Y(a,b)}.$$

For $j=0$, noting that $\epsilon \nu^{\frac13}t\leq 1$ for $t\in \mathrm{I}_{0}$, we use Lemma \ref{space-time thete} to deduce
\begin{align}\label{esti: omjwithj=0}
&\|e^{\epsilon_{0} \nu^{\frac13}t}\theta_0\|^2_{X(0,T)}\\
\lesssim &\|\theta^{\mathrm{in}}\|^2_{L^2}+\nu^{-\frac13}|k|^{-\frac23}(\|e^{\epsilon_{0} \nu^{\frac13}t}g_{1}\|^2_{L^2(\mathrm{I}_{0}; L^2)}+\|e^{\epsilon_{0} \nu^{\frac13}t}G_{0}\|^2_{L^2(\mathrm{I}_{0}; L^2)})+\nu^{-1}\|e^{\epsilon_{0} \nu^{\frac13}t}g_2\|^2_{L^2(\mathrm{I}_{0}; L^2)}\notag\\
\lesssim &\|\theta^{\mathrm{in}}\|^2_{L^2}+\nu^{-\frac13}|k|^{-\frac23}\|g_1\|^2_{L^2(\mathrm{I}_{0}; L^2)}+ \nu^{-1}\|g_2\|^2_{L^2(\mathrm{I}_{0}; L^2)}+\nu^{-\frac13}|k|^{-\frac23}\|G_{0}\|^2_{L^2(\mathrm{I}_{0}; L^2)}.\notag
\end{align}

For $1 \leq j\leq N$, similarly, we have
\begin{align}\label{esti: omjwithjgeq1}
\|e^{\epsilon_{0} \nu^{\frac13}t}\theta_j\|^2_{X(t_j,T)}
\lesssim \nu^{-\frac13}|k|^{-\frac23}\|g_1\|^2_{L^2(\mathrm{I}_{j}; L^2)}+\nu^{-1}\|g_2\|^2_{L^2(\mathrm{I}_{j}; L^2)}+\nu^{-\frac13}|k|^{-\frac23}\|G_{j}\|^2_{L^2(\mathrm{I}_{j}; L^2)}.
\end{align}

It follows from \eqref{esti: Ut-Us Linfty} that
\begin{align*}
\n{U(t_{j'+1},y)-U(t,y)}_{L^{\infty}(\mathrm{I}_{j}; L^{\infty})}\lesssim \nu \|t-t_{j'+1}\|_{L^\infty(\mathrm{I}_{j})}\|U^{in}\|_{H^4}\lesssim \nu^{\frac23}(j-j'+1),
\end{align*}
which, together with the definition of $G_{j}$, implies
\begin{align}\label{esti:F1+F2}
\nu^{-\frac16}|k|^{-\frac13}\|G_{j}\|_{L^2(\mathrm{I}_{j}; L^2)}
\leq &\nu^{-\frac16}|k|^{\frac23}\sum_{j'=0}^{j}\|U(t_{j'+1},y)-U(t,y)\|_{L^{\infty}(\mathrm{I}_{j}; L^{\infty})}\n{\theta_{j'}}_{L^2(\mathrm{I}_{j};L^2)}\\
\lesssim & \nu^{\frac{1}{2}}|k|^{\frac23}\sum_{j'=0}^{j}(j-j'+1)\|\theta_{j'}\|_{L^2(\mathrm{I}_{j}; L^2)}.\notag
\end{align}

Due to $\nu^{\frac13}(t-t_{j'})\geq j-j'$ for $t\in \mathrm{I}_{j}$, we have
\begin{align*}
\|\theta_{j'}\|_{L^2(\mathrm{I}_{j}; L^2)}
\lesssim  e^{-\epsilon_{0} (j-j')}\n{e^{\epsilon_{0} \nu^{\frac{1}{3}}(t-t_{j'})} \theta_{j'}}_{L^2(\mathrm{I}_{j};L^2)}.
\end{align*}
Inserting it into \eqref{esti:F1+F2} and using the Cauchy-Schwarz inequality, we obtain
\begin{align}\label{esti:G1+G2}
\nu^{-\frac13}|k|^{-\frac23}\|G_{j}\|^2_{L^2(\mathrm{I}_{j}; L^2)}
 \lesssim & \nu^{\frac{1}{3}}  \lrs{\sum_{j'=0}^{j}(j-j'+1)  e^{-\epsilon_{0} (j-j') }\nu^{\frac13}|k|^{\frac23}\|e^{\epsilon_{0} \nu^{\frac{1}{3}}(t-t_{j'})}\theta_{j'}\|_{L^2(\mathrm{I}_{j};L^2)}}^2\\
 \lesssim&  \nu^{\frac{1}{3}}  \sum_{j'=0}^{j} e^{-\epsilon_{0} (j-j') }\|e^{\epsilon_{0} \nu^{\frac{1}{3}}(t-t_{j'})}\theta_{j'}\|_{Y(\mathrm{I}_{j})}^2. \notag
\end{align}

Substituting \eqref{esti:G1+G2} into \eqref{esti: omjwithj=0} and \eqref{esti: omjwithjgeq1}, we get
\begin{align}\label{esti: theta0}
\|e^{\epsilon_{0} \nu^{\frac13}t}\theta_{0}\|^2_{X(0,T)}\lesssim & \|\theta^{\mathrm{in}}\|^2_{L^2}+\nu^{-\frac13}|k|^{-\frac23}\|g_1\|^2_{L^2(\mathrm{I}_{0}; L^2)}\\
&+ \nu^{-1}\|g_2\|^2_{L^2(\mathrm{I}_{0}; L^2)}+\nu^{\frac{1}{3}}\|e^{\epsilon_{0} \nu^{\frac{1}{3}}t}\theta_{0}\|^2_{Y(\mathrm{I}_{0})}, \notag
\end{align}
and
\begin{align}\label{esti: theta1}
\|e^{\epsilon_{0} \nu^{\frac13}(t-t_{j})}\theta_{j}\|^2_{X(t_j,T)}\lesssim & \nu^{-\frac13}|k|^{-\frac23}\|g_1\|^2_{L^2(\mathrm{I}_{j}; L^2)}+\nu^{-1}\|g_2\|^2_{L^2(\mathrm{I}_{j}; L^2)}\\
&+\nu^{\frac{1}{3}}\sum_{j'=0}^{j}  e^{-\epsilon_{0}(j-j')} \|e^{\epsilon_{0} \nu^{\frac{1}{3}}(t-t_{j'})}\theta_{j'}\|^2_{Y(\mathrm{I}_{j})}. \notag
\end{align}

Let
\begin{align*}
H_{j}=\sum_{j'=0}^{j}  e^{-\epsilon_{0} (j-j')} \|e^{\epsilon_{0} \nu^{\frac{1}{3}}(t-t_{j'})}\theta_{j'}\|^2_{Y(\mathrm{I}_{j})}.
\end{align*}
It follows from \eqref{esti: theta0} and \eqref{esti: theta1} that
\begin{align*}
H_{j}
\lesssim
& e^{-\epsilon_{0} j} \|\theta^{\mathrm{in}}\|^2_{L^2}+ \sum_{j'=0}^{j}  e^{-\epsilon_{0} (j-j')} \Big[\nu^{-\frac13}|k|^{-\frac23}\|g_1\|^2_{L^2(\mathrm{I}_{j}; L^2)}+
 \nu^{-1}\|g_2\|^2_{L^2(\mathrm{I}_{j}; L^2)}+\nu^{\frac{1}{3}}H_{j'}\Big].
\end{align*}

Choosing $\epsilon\leq \frac{\epsilon_{0}}{4}$ and noticing that
$$e^{2\epsilon j}e^{-\epsilon_0(j-j')}=e^{-(\epsilon_0-2\epsilon)(j-j')}e^{2\epsilon j'},$$
we then apply Minkowski's inequality to obtain
\begin{align*}
\sum_{j=0}^{N}e^{2\epsilon j}H_{j}
\lesssim&\|\theta^{\mathrm{in}}\|^2_{L^2}+ \sum_{j=0}^{N}\sum_{j'=0}^{j}  e^{-(\epsilon_{0}-2\epsilon )(j-j')}e^{2\epsilon j' } \Big[ \nu^{-\frac13}|k|^{-\frac23}\|g_1\|^2_{L^2(\mathrm{I}_{j'}; L^2)}\\
&\qquad \qquad + \nu^{-1}\|g_2\|^2_{L^2(\mathrm{I}_{j'}; L^2)}+\nu^{\frac{1}{3}}H_{j'}\Big]\\
\lesssim& \|\theta^{\mathrm{in}}\|^2_{L^2}+\sum_{j'=0}^{N} e^{2\epsilon j' } \Big[\nu^{-\frac13}|k|^{-\frac23}\|g_1\|^2_{L^2(\mathrm{I}_{j'}; L^2)}+ \nu^{-1}\|g_2\|^2_{L^2(\mathrm{I}_{j'}; L^2)}+\nu^{\frac{1}{3}}H_{j'}\Big]\\
\lesssim & \|\theta^{\mathrm{in}}\|^2_{L^2}+\nu^{-\frac13}|k|^{-\frac23}\|g_1\|^2_{L^2( 0,T; L^2)}+ \nu^{-1}\|g_2\|^2_{L^2(0,T ; L^2)}+\nu^{\frac13} \sum_{j=0}^{N} e^{2\epsilon  j }H_{j},
\end{align*}
which implies that
\begin{align}\label{esti: sumDj}
\sum_{j=0}^{N}e^{2\epsilon  j}H_{j}
 \lesssim \|\theta^{\mathrm{in}}\|^2_{L^2}+\nu^{-\frac13}|k|^{-\frac23}\|g_1\|^2_{L^2( 0,T; L^2)}+ \nu^{-1}\|g_2\|^2_{L^2(0,T ; L^2)}.
\end{align}

Now, we go back to the estimate of $\|e^{\epsilon_{0} \nu^{\frac13}t}\theta\|_{X(0,T)}$.

For $\|e^{\epsilon_{0} \nu^{\frac13}t}\theta\|_{L^\infty(0,T;L^2)}$,  due to $(0,T)=\sum_{j=0}^{N}\mathrm{I}_{j}$ and $\theta=\sum_{j'=0}^{j}\theta_{j'}$ for $t\in \mathrm{I}_{j}$, we obtain
\begin{align*}
\|e^{\epsilon \nu^{\frac13} t}\theta\|^2_{L^\infty(0,T;L^2)}= \sup_{0\leq j\leq N}\|e^{\epsilon \nu^{\frac13} t}\theta\|^2_{L^\infty(\mathrm{I}_{j};L^2)}\leq \sup_{0\leq j\leq N} \sum_{j'=0}^{j}\|e^{\epsilon \nu^{\frac13} t}\theta_{j'}\|^2_{L^\infty(\mathrm{I}_{j};L^2)}.
\end{align*}
Inserting \eqref{esti: theta0} and \eqref{esti: theta1} into above estimate and using \eqref{esti: sumDj}, we deduce
\begin{align}
\|e^{\epsilon \nu^{\frac13} t}\theta\|^2_{L^\infty(0,T;L^2)}\lesssim &\|\theta^{\mathrm{in}}\|^2_{L^2}+ \sup_{0\leq j\leq N} \sum_{j'=0}^{j}\nu^{-\frac13}|k|^{-\frac23}\|g_1\|^2_{L^2(\mathrm{I}_{j'}; L^2)} \label{nk2leq1thetaLinfty}\\
&+  \sup_{0\leq j\leq N} \sum_{j'=0}^{j} \nu^{-1}\|g_2\|^2_{L^2(\mathrm{I}_{j'}; L^2)}+\nu^{\frac{1}{3}}\sup_{0\leq j\leq N}\sum_{j'=0}^{j}H_{j'}\notag\\
\lesssim & \|\theta^{\mathrm{in}}\|^2_{L^2}+ \nu^{-\frac13}|k|^{-\frac23}\|g_1\|^2_{L^2(0,T; L^2)}+ \nu^{-1}\|g_2\|^2_{L^2(0,T; L^2)}. \notag
\end{align}

For $\|e^{\epsilon_{0} \nu^{\frac13}t}\theta\|_{L^2(0,T;L^2)}$, similarly, we have
\begin{align}\label{esti:omX0T}
\|e^{\epsilon \nu^{\frac13} t}\theta\|^2_{L^2(0,T;L^2)}=&\sum_{j=0}^{N}\int_{\mathrm{I}_{j}}\|e^{\epsilon \nu^{\frac13}t}\theta\|^2_{L^2} dt
\lesssim\sum_{j=0}^{N}e^{2\epsilon j}\int_{\mathrm{I}_{j}}\bbn{\sum_{j'=0}^{j}\theta_{j'}}^2_{L^2} dt.
\end{align}
Using H\"older's inequality and the fact that $\nu^{\frac13}(t-t_{j'})\geq j-j'$ for $t\in \mathrm{I}_{j}$, we obtain
\begin{align}\label{equ:Ij,sum,j,Y}
\int_{\mathrm{I}_{j}}\lrs{\sum_{j'=0}^{j}\n{\theta_{j'}}_{L^2}}^2 dt\leq&\int_{\mathrm{I}_{j}}\sum_{j'=0}^{j}e^{-\epsilon_{0} \nu^{\frac13}(t-t_{j'})}\n{e^{\epsilon_{0} \nu^{\frac13}(t-t_{j'})}\theta_{j'}}^2_{L^2}\sum_{j'=0}^{j}e^{-\epsilon_{0} \nu^{\frac13}(t-t_{j'})} dt\\
\lesssim& \int_{\mathrm{I}_{j}}\sum_{j'=0}^{j} e^{-\epsilon_{0} (j-j')} \n{e^{\epsilon_{0} \nu^{\frac13}(t-t_{j'})}\theta_{j'}}^2_{L^2} dt\lesssim \nu^{-\frac13}|k|^{-\frac23} H_{j}.\notag
\end{align}
Putting together \eqref{esti: sumDj}, \eqref{esti:omX0T}, and \eqref{equ:Ij,sum,j,Y}, we arrive at
\begin{align*}
\nu^{\frac13}|k|^{\frac23}\|e^{\epsilon \nu^{\frac13} t}\theta\|^2_{L^2(0,T;L^2)}\lesssim \|\theta^{\mathrm{in}}\|^2_{L^2}+\nu^{-\frac13}|k|^{-\frac23}\|e^{\epsilon \nu^{\frac13} t}g_1\|^2_{L^2( 0,T; L^2)}+ \nu^{-1}\|e^{\epsilon \nu^{\frac13} t}g_2\|^2_{L^2(0,T ; L^2)}.
\end{align*}
This estimate together with \eqref{nk2geq1hteta} and \eqref{nk2leq1thetaLinfty} completes the proof of \eqref{esti: theta}.
\end{proof}

\section{Nonlinear stability}\label{nonlinear tran thre}
In this section, we will use the space-time estimates obtained in Section \ref{sec: space-time estimate} to derive the transition threshold of the monotone shear flow, i.e., the proof of Theorem \ref{Th: tran thre}.

Recall
\begin{equation*}
E_k=\left\{
\begin{aligned}
&\|\om_0\|_{L^\infty L^2}, \quad k=0,\\
&|k|\| e^{\epsilon_0 \nu^{\frac13}t} u_k\|_{L^2L^2}+|k|^{\frac12}\|e^{\epsilon_0 \nu^{\frac13}t} u_k\|_{L^\infty L^\infty}\\
&+\|e^{\epsilon_0 \nu^{\frac13}t} \sqrt{1-(2y-1)^2} \om_k\|^2_{L^\infty L^2}+\nu^{\frac14}|k|^{\frac12}\|e^{\epsilon_0 \nu^{\frac13}t} \om_k\|_{L^2L^2}, \quad k\neq 0.
\end{aligned}
\right.
\end{equation*}
and
\begin{equation*}
G_k=\left\{
\begin{aligned}
&\|\theta_0\|_{L^\infty L^2}, \quad k=0,\\
&|k|^{\frac13}\|e^{\epsilon_0 \nu^{\frac13}t} \theta_k\|_{L^\infty L^2}+\nu^{\frac16}|k|^{\frac23}\|e^{\epsilon_0 \nu^{\frac13}t} \theta_k\|_{L^2L^2}, \quad k\neq 0,
\end{aligned}
\right.
\end{equation*}
\begin{proof}[Proof of Theorem \ref{Th: tran thre}]
The equation \eqref{equ: omli11} can be rewritten as
\begin{equation}\label{equ: omtheta}
\left\{
\begin{aligned}
&\pat\om-\nu\De \om+U\pa_x\om-\pa^2_y U\pa_x\psi=-\operatorname{curl}(u\cdot\na u)+\pa_x \theta,\\
&\pat \theta-\nu \De\theta+U\pa_x\theta=-u\cdot\na \theta.
\end{aligned}
\right.
\end{equation}
Taking the Fourier transform with respect to $x$-variable on the equation of \eqref{equ: omtheta}, we obtain
\begin{align*}
\pa_t \om_k-\nu (\pa^2_y-k^2) \om_k+ikU \om_k-ik\pa^2_y U\psi_k=&\pa_y(u\cdot\na u^{(1)})_k-ik(u\cdot\na u^{(2)})_k-ik \theta \\
:=&-\pa_y (f^{1,1}_k+f^{2,1}_k)-ik(f^{1,2}_k+f^{2,2}_k)-ik \theta ,\\
\pat \theta_{k}-\nu (\pa^2_y-k^2)\theta_{k}+ik U\theta_{k}:= &-ik g_{k}^{1}-\pa_y g_{k}^{2}.
\end{align*}
where
\begin{align*}
&f^{1,1}_k=i\sum_{l\in\Z} u^{(1)}_l(t,y)(k-l)u^{(1)}_{k-l}(t,y), \quad f^{1,2}_k=i\sum_{l\in\Z} u^{(1)}_l(t,y)(k-l)u^{(2)}_{k-l}(t,y),\\
&f^{2,1}_k=\sum_{l\in \Z} u^{(2)}_l(t,y) \pa_y u^{(1)}_{k-l}(t,y), \qquad\quad f^{2,2}_k=\sum_{l\in \Z} u^{(2)}_l(t,y) \pa_y u^{(2)}_{k-l}(t,y),\\
&g_{k}^{1}=\sum_{l\in\Z}u^{1}_{k-1}\theta_{l},\qquad \qquad \qquad g_{k}^{2}=\sum_{l\in\Z}  (u^{2}_{k-l}\theta_{l}).
\end{align*}

The conditions $\operatorname{div} u=0$ and $P_0(u^{(1)}\pa_x u^{(1)})=0,$ together with the equation \eqref{pertu} imply that $u^{(1)}_0$ satisfies
\begin{align}\label{new,equ:u01,def}
(\pa_t -\nu \pa^2_y) u^{(1)}_0(t,y)=-\sum_{l\in \Z\setminus \{0\}} u^{(2)}_l \pa_y u^{(1)}_{-l}(t,y):=-f_0^{2,1}(t,y).
\end{align}
Taking inner products with $\pa_y^2u_0^{(1)}$ and using the integration by parts, we have
\begin{align*}
\pa_t \|\pa_y u^{(1)}_0\|^2_{L^2}+\nu \|\pa^2_y u^{(1)}_0\|^2_{L^2}=\lan f_0^{2,1}, \pa^2_y u^{(1)}_0\ran \lesssim \nu^{-1}\|f_0^{2,1}\|^2_{L^2}+\frac12\nu\|\pa^2_y u^{(1)}_0\|^2_{L^2},
\end{align*}
which implies
\begin{align}\label{esti: E0}
E^2_0\lesssim \nu^{-1}\|f_0^{2,1}\|^2_{L^2L^2}+\|\om^{\mathrm{in}}_0\|^2_{L^2},
\end{align}
where we used the fact that $\pa_yu_0^{(1)}(t,y)=\om_0(t,y)$.

For $E_k$ with $k\neq 0$, it follows from space-time estimates \eqref{equ:low frequency} that
\begin{equation}\label{esti: Ek}
\begin{aligned}
E^2_{k}\lesssim& \nu^{-1} \|e^{\epsilon_0 \nu^{\frac13}t} f^{m,n}_k\|^2_{L^2 L^2} +\min\{\nu^{-\frac13}|k|^{\frac43}, \nu^{-1} \}\|\theta_{k}\|^2_{L^2L^2}\\
& +\|u^{\mathrm{in}}_k\|^2_{H^1_y}+|k|^{-2}\|\pa_y\om^{\mathrm{in}}_{k}\|^2_{L^2}, \qquad 1\leq m,n\leq 2.
\end{aligned}
\end{equation}

By H\"older's inequality, for $n=1,2$, we have
\begin{align}\label{esti: f1}
\|e^{\epsilon_0 \nu^{\frac13}t}f^{1,n}\|_{L^2L^2}\leq \sum_{l\in\Z}\|e^{\epsilon_0 \nu^{\frac13}t}u^{(1)}_l\|_{L^\infty L^\infty}\|(k-l)e^{\epsilon_0 \nu^{\frac13}t}u^{(n)}_{k-l}\|_{L^2L^2}\leq \sum_{l\in \Z} E_l E_{k-l}.
\end{align}

Using Hardy's inequality, $\pa_{y}u^{(2)}=-iku^{(1)}$ and \eqref{esti: weightpayu}, for $n=1,2$, we deduce
\begin{equation}
\begin{aligned}\label{esti: f2}
\|e^{\epsilon_0 \nu^{\frac13}t}f^{2,n}\|_{L^2L^2}
\leq & \sum_{l\in\Z} \Big\|\frac{e^{\epsilon_0 \nu^{\frac13}t}u^{(2)}_{l}}{\sqrt{1-(2y-1)^2}}\Big\|_{L^2 L^\infty}\|e^{\epsilon_0 \nu^{\frac13}t}\sqrt{1-(2y-1)^2} \pa_y u^{(n)}_{k-l}\|_{L^\infty L^2}\\
\lesssim &\sum_{l\in\Z} \|e^{\epsilon_0 \nu^{\frac13}t}\pa_y u^{(2)}_l\|_{L^2L^2}\|e^{\epsilon_0 \nu^{\frac13}t}\sqrt{1-(2y-1)^2}\om_{k-l}\|_{L^\infty L^2}\\
\lesssim & \sum_{l\in\Z}\|e^{\epsilon_0 \nu^{\frac13}t}l u^{(1)}_l\|_{L^2L^2}E_{k-l} \lesssim  \sum_{l\in\Z}E_{l}E_{k-l}.
\end{aligned}
\end{equation}

Inserting \eqref{esti: f1}, \eqref{esti: f2} into \eqref{esti: Ek}, and \eqref{esti: f2} into \eqref{esti: E0}, we arrive at
\begin{align}\label{esti: Ek1}
E_k\lesssim \nu^{-\frac12}\sum_{l\in\Z} E_{l}E_{k-l}+\nu^{-\frac13}G_{k}+\|u^{\mathrm{in}}_k\|_{H^1_y}+|k|^{-1}\|\pa_y\om^{\mathrm{in}}_{k}\|_{L^2}, \quad k\in\Z.
\end{align}

For $G_{0}$, due to $P_{0}(u^{1}_{l}\pa_{x}\theta_{-l})=0$, we have
\begin{align*}
\pa_t \theta_{0}-\nu \pa^2_y \theta_{0}= -\sum_{l\in \Z\setminus \{0\}} u^{(2)}_l \pa_y \theta_{-l}:=\pa_y g_{0}^{2}.
\end{align*}
Test the above equation with $\theta_{0}$ gives
\begin{align*}
\pa_t \|\theta_{0}\|^2_{L^2}+2\nu\|\pa_y \theta_{0}\|^2_{L^2}=2\lan \pa_y g_{0}^{2}, \theta_{0} \ran\leq C\nu^{-1}\|g_{0}^{2}\|^2_{L^2}+\frac12\nu\|\pa_y \theta_{0}\|^2_{L^2},
\end{align*}
which implies
\begin{align}\label{esti:H0}
G^2_{0}\leq \|\theta_{0}\|^2_{L^\infty L^2}+\nu \|\pa_y \theta_{0}\|^2_{L^2L^2}\leq \|\theta^{\mathrm{in}}_{0}\|_{L^2}+C\nu^{-1}\|g_{0}^{2}\|^2_{L^2L^2}.
\end{align}

For $G_k$ with $k\neq 0$, we use \eqref{esti: theta} to obtain
\begin{align}\label{esti: Hk}
G^2_{k}\lesssim& |k|^{\frac13}\|\theta^{\mathrm{in}}_k\|^2_{L^2_y}+\min\{\nu^{-\frac13}|k|^{2}, \nu^{-1} |k|^{\frac23}\}\|e^{\epsilon_0 \nu^{\frac13}t} g^1_k\|^2_{L^2L^2} +\nu^{-1} |k|^{\frac23} \|e^{\epsilon_0 \nu^{\frac13}t} g^2_k\|^2_{L^2L^2}.
\end{align}

For $g^1_k$, it is noticed that
\begin{align*}
\|e^{\epsilon_0 \nu^{\frac13}t} g^1_k\|_{L^2L^2}
\leq &\|e^{\epsilon_0 \nu^{\frac13}t} u^{1}_{k}\|_{L^2L^\infty}\|e^{\epsilon_0 \nu^{\frac13}t} \theta_{0}\|_{L^\infty L^2}+\|e^{\epsilon_0 \nu^{\frac13}t} u^{1}_{0}\|_{L^\infty L^\infty}\|e^{\epsilon_0 \nu^{\frac13}t} \theta_{k}\|_{L^2 L^2}\\
&+\sum_{l\in \Z\setminus \{0,k\}} \|e^{\epsilon_0 \nu^{\frac13}t} u^{1}_{k-l}\theta_{l}\|_{L^2L^2}.
\end{align*}

By the definition of $E_{k}$ and $G_{k}$, we have
\begin{align}
&\|e^{\epsilon_0 \nu^{\frac13}t} u^{1}_{k}\|_{L^2L^\infty}\|e^{\epsilon_0 \nu^{\frac13}t} \theta_{0}\|_{L^\infty L^2}\leq \|e^{\epsilon_0 \nu^{\frac13}t} u^{1}_{k}\|^{\frac12}_{L^2L^2}\|e^{\epsilon_0 \nu^{\frac13}t} \om_{k}\|^{\frac12}_{L^2L^2}G_{0}\leq \nu^{-\frac{1}{8}}|k|^{-\frac34}E_{k}G_{0}, \label{esti:EkH0}\\
&\|e^{\epsilon_0 \nu^{\frac13}t} u^{1}_{0}\|_{L^\infty L^\infty}\|e^{\epsilon_0 \nu^{\frac13}t} \theta_{k}\|_{L^2 L^2} \leq \|e^{\epsilon_0 \nu^{\frac13}t} \om_{0}\|_{L^\infty L^2}G_{k}\leq \nu^{-\frac{1}{6}} |k|^{-\frac23} E_{0}G_{k}.\label{esti:HkE0}
\end{align}

For $l\in \Z\setminus \{0,k\}$, we divide the summation into
\begin{align}\label{sum}
\sum_{l\in \Z\setminus \{0,k\}} \|e^{\epsilon_0 \nu^{\frac13}t} u^{1}_{k-l}\theta_{l}\|_{L^2L^2}
\leq \sum_{\substack{l\in \Z\setminus \{0,k\},\\ |k-l|\leq |k|/2} } \|e^{\epsilon_0 \nu^{\frac13}t} u^{1}_{k-l}\theta_{l}\|_{L^2L^2}+\sum_{\substack{l\in \Z\setminus \{0,k\},\\ |k-l|\geq |k|/2}} \|e^{\epsilon_0 \nu^{\frac13}t} u^{1}_{k-l}\theta_{l}\|_{L^2L^2}.
\end{align}
For $|k-l|\leq |k|/2$, we have $|l|^{-1}\leq 2|k|^{-1}$, which implies
\begin{align}\label{k-l<k}
\|e^{\epsilon_0 \nu^{\frac13}t} u^{1}_{k-l}\theta_{l}\|_{L^2L^2}\leq &\|e^{\epsilon_0 \nu^{\frac13}t} u^{1}_{k-l}\|_{L^\infty L^\infty}\|e^{\epsilon_0 \nu^{\frac13}t}\theta_{l}\|_{L^2L^2}\\
\leq &\nu^{-\frac16} |k-l|^{-\frac12}|l|^{-\frac23}E_{k-l}  G_{l}
\leq  \nu^{-\frac16} |k|^{-\frac23}E_{k-l}  G_{l}.\notag
\end{align}
For $|k-l|\geq |k|/2$, we have $|k-l|^{-1}\leq 2|k|^{-1}$. It follows from interpolation inequality that 
\begin{align}\label{k-l>k}
\|e^{\epsilon_0 \nu^{\frac13}t} u^{1}_{k-l}\theta_{l}\|_{L^2L^2}\leq &\|e^{\epsilon_0 \nu^{\frac13}t} u^{1}_{k-l}\|_{L^2L^\infty} \|e^{\epsilon_0 \nu^{\frac13}t}\theta_{l}\|_{L^\infty L^2}\\
\lesssim & \|e^{\epsilon_0 \nu^{\frac13}t} u^{1}_{k-l}\|^{\frac12}_{L^2L^2}\|e^{\epsilon_0 \nu^{\frac13}t} \om_{k-l}\|^{\frac12}_{L^2L^2}  |l|^{-\frac13} G_{l}\notag \\
\lesssim & \nu^{-\frac18}|k-l|^{-\frac34}|l|^{-\frac13}E_{k-l}G_{l}\lesssim \nu^{-\frac18}|k|^{-\frac34}E_{k-l}G_{l}.\notag
\end{align}

Inserting \eqref{k-l<k} and \eqref{k-l>k} into \eqref{sum}, we further obtain
\begin{align*}
&\min\{\nu^{-\frac16}|k|,\nu^{-\frac12}|k|^{\frac13}\}\sum_{l\in \Z\setminus \{0,k\}} \|e^{\epsilon_0 \nu^{\frac13}t} u^{1}_{k-l}\theta_{l}\|_{L^2L^2}\\
\lesssim & \sum_{l\in \Z\setminus \{0,k\}} \min\{(\nu^{-\frac13}|k|^{\frac13}+\nu^{-\frac{7}{24}}|k|^{\frac14} ),  (\nu^{-\frac23}|k|^{-\frac13}+\nu^{-\frac58}|k|^{-\frac{5}{12}})\} E_{k-l} G_{l} \\
\lesssim &\nu^{-\frac12}\sum_{l\in \Z\setminus \{0,k\}} E_{k-l}G_{l},
\end{align*}
which together with \eqref{esti:EkH0} and \eqref{esti:HkE0}, implies
\begin{align}\label{esti:g1}
\min\{\nu^{-\frac16}|k|, \nu^{-\frac12} |k|^{\frac13}\}\|e^{\epsilon_0 \nu^{\frac13}t} g^1_k\|_{L^2L^2}\lesssim \nu^{-\frac12}\sum_{l\in \Z} E_{k-l}G_{l}.
\end{align}

For $g^2_k$, noticing that $\pa_y u^2_{k-l}=-i(k-l) u^2_{k-l}$, we use Hardy's inequality to obtain
\begin{align*}
\|e^{\epsilon_0 \nu^{\frac13}t} g^2_k\|_{L^2L^2}\leq &\sum_{l\in\Z} \|e^{\epsilon_0 \nu^{\frac13}t}u^2_{k-l}\|_{L^2 L^\infty}\|e^{\epsilon_0 \nu^{\frac13}t}\theta_{l}\|_{L^\infty L^2}\\
\leq &\sum_{l\in \Z}\|e^{\epsilon_0 \nu^{\frac13}t}u^2_{k-l}\|^{\frac12}_{L^2L^2}\|e^{\epsilon_0 \nu^{\frac13}t}\pa_y u^2_{k-l}\|^{\frac12}_{L^2L^2} \|e^{\epsilon_0 \nu^{\frac13}t}\theta_{l}\|_{L^\infty L^2}\\
= &\sum_{l\in \Z}\|e^{\epsilon_0 \nu^{\frac13}t}u^2_{k-l}\|^{\frac12}_{L^2L^2}|k-l|^{\frac12}\|e^{\epsilon_0 \nu^{\frac13}t} u^1_{k-l}\|^{\frac12}_{L^2L^2} \|e^{\epsilon_0 \nu^{\frac13}t}\theta_{l}\|_{L^\infty L^2}\\
\leq & \sum_{l\in \Z\setminus \{0,k\}} |k-l|^{-\frac12}|l|^{-\frac13} E_{k-l}G_{l}+|k|^{-\frac12}E_{k}G_{0},
\end{align*}
which implies
\begin{align}\label{esti:g2}
\nu^{-\frac12}|k|^{\frac13}\|e^{\epsilon_0 \nu^{\frac13}t} g^2_k\|_{L^2L^2}\leq \nu^{-\frac12}\sum_{l\in \Z} E_{k-l}G_{l}.
\end{align}

Substituting \eqref{esti:g1} and \eqref{esti:g2} into \eqref{esti: Hk} and together with \eqref{esti:H0}, we deduce
\begin{align}\label{esti: Hk1}
H_k\lesssim \nu^{-\frac12}\sum_{l\in \Z}E_{k-l}  H_{l} +|k|^{\frac13}\|\theta^{\mathrm{in}}_{k}\|_{L^2}, \quad k\in \Z.
\end{align}

Summing up with $k$ in \eqref{esti: Ek1} and \eqref{esti: Hk1}, we further obtain
\begin{align*}
\sum_{k\in\Z} E_k\lesssim \nu^{-\frac12}\sum_{k\in\Z}\sum_{l\in \Z} E_{l}E_{k-l} +\nu^{-\frac13}\sum_{k\in \Z} G_{k} +\sum_{k\in\Z}(\|u^{\mathrm{in}}_k\|_{H^1_y}+|k|^{-1}\|\pa_y\om^{\mathrm{in}}_{k}\|_{L^2}),
\end{align*}
and
\begin{align*}
\sum_{k\in\Z} G_k\lesssim \nu^{-\frac12}\sum_{k\in\Z}\sum_{l\in \Z}E_{k-l}  G_{l} +\sum_{k\in\Z}|k|^{\frac13}\|\theta^{\mathrm{in}}_{k}\|_{L^2}.
\end{align*}
Thanks to the initial data condition
\begin{align*}
\|u^{\mathrm{in}}\|_{H^2}\leq c\nu^{\frac12}, \qquad \|\theta^{\mathrm{in}}\|_{L^2}+\||D_{x}|\theta^{\mathrm{in}}\|_{L^2}\leq c\nu^{\frac56},
\end{align*}
we finally derive
\begin{align*}
\sum_{k\in\Z}E_k\lesssim \nu^{\frac12}, \qquad \sum_{k\in \Z}G_{k}\lesssim \nu^{\frac56},
\end{align*}
which completes the proof of Theorem \ref{Th: tran thre}.
\end{proof}
\appendix
\section{Auxiliary estimates}
\begin{lemma}\label{lemma: sinh}
There holds
\begin{align}
&\bbn{\frac{\sinh{k(1-y)}}{\sinh k}}_{L^2}\lesssim |k|^{-\frac12}, \qquad \bbn{\frac{\sinh{(ky)}}{\sinh k}}_{L^2}\lesssim |k|^{-\frac12}, \label{esti:sinh}\\
&\bbn{\frac{\cosh{k(1-y)}}{\sinh k}}_{L^2}\lesssim |k|^{-\frac12}, \qquad \bbn{\frac{\cosh{(ky)}}{\sinh k}}_{L^2}\lesssim |k|^{-\frac12}. \label{esti:cosh}
\end{align}
\end{lemma}
\begin{lemma}
Let $(\pa^2_y- k^2)\psi=\om$, $\psi(\pm1)=0$. Then, we have
\begin{align}\label{esti: weightpayu}
\|\sqrt{1-(2y-1)^2} \pa_y u\|^2_{L^2}\leq \|\sqrt{1-(2y-1)^2}\om\|^2_{L^2}.
\end{align}
\end{lemma}

\begin{lemma}\label{lemma:heat,dirichlet,H2}
Let $U$ be the solution to the equation \eqref{equ: U} with the initial datum $U^{\mathrm{in}}$ satisfying the condition \eqref{equ:condition,shear,flow,M}. Then the following estimate holds
\begin{align}
\|U(t,y)-U(s,y)\|_{ L^\infty_{y}}\lesssim & \nu(t-s)\|U^{\mathrm{in}}\|_{H^4}.\label{esti: Ut-Us Linfty}
\end{align}
\end{lemma}
The proof of Lemma \ref{lemma:heat,dirichlet,H2} can be found in \cite{Li-Shen-Zhang}.

\medskip
\noindent
\textbf{Acknowledgements.}  Q. Chen was partially supported by National Natural Science Foundation of China (Grant No: 12471149). Z. Li was partially supported by the Postdoctoral Fellowship Program of CPSF (Grant No: GZC20240123 and 2025M773073).

\noindent\textbf{Data Availability Statement:}
Data sharing is not applicable to this article as no datasets were generated or analysed during the current study.

\noindent\textbf{Conflict of Interest:}
The authors declare that they have no conflict of interest.

\end{document}